\documentclass[11pt,letterpaper]{amsart}
\usepackage{amsmath,amssymb,amsthm}
\usepackage{enumerate}
\renewcommand{\phi}{\varphi}
\DeclareMathOperator{\Tr}{Tr}
\newcommand{\C}{{\mathbb C}}
\newcommand{\F}{{\mathbb F}}
\newcommand{\Q}{{\mathbb Q}}
\newcommand{\Z}{{\mathbb Z}}
\newcommand{\card}[1]{\left|{#1}\right|}
\newcommand{\conj}[1]{\overline{#1}}
\newcommand{\sums}[1]{\sum_{\substack{#1}}}
\newcommand{\val}{v_p}
\newcommand{\ft}[1]{\widehat{#1}}
\newcommand{\Fp}{\F_p}
\newcommand{\Fu}{F^*}
\newcommand{\mchars}{\ft{\Fu}}
\newcommand{\gr}{\C[\Fu]}
\newcommand{\fr}{\C^{\mchars}}
\newcommand{\Wfd}{W_{F,d}}
\newcommand{\Wfdu}{W_{F,d}(u)}
\newcommand{\Wu}{W_u}
\newcommand{\Wc}{\conj{W}}
\newcommand{\Psip}[1]{\Psi^{(#1)}}
\newtheorem{theorem}{Theorem}[section]
\newtheorem{proposition}[theorem]{Proposition}
\newtheorem{lemma}[theorem]{Lemma}
\newtheorem{corollary}[theorem]{Corollary}
\newtheorem{conjecture}[theorem]{Conjecture}
\theoremstyle{remark}
\newtheorem{remark}[theorem]{Remark}
\title{Divisibility of Weil Sums of Binomials}
\author{Daniel J. Katz}
\address{Department of Mathematics, California State University, Northridge, \: United States}
\date{original version: 29 July 2014; this version: 16 March 2015}
\begin{document}
\begin{abstract}
Consider the Weil sum $W_{F,d}(u)=\sum_{x \in F} \psi(x^d+u x)$, where $F$ is a finite field of characteristic $p$, $\psi$ is the canonical additive character of $F$, $d$ is coprime to $|F^*|$, and $u \in F^*$.
We say that $W_{F,d}(u)$ is three-valued when it assumes precisely three distinct values as $u$ runs through $F^*$: this is the minimum number of distinct values in the nondegenerate case, and three-valued $W_{F,d}$ are rare and desirable.
When $W_{F,d}$ is three-valued, we give a lower bound on the $p$-adic valuation of the values.
This enables us to prove the characteristic $3$ case of a 1976 conjecture of Helleseth: when $p=3$ and $[F:{\mathbb F}_3]$ is a power of $2$, we show that $W_{F,d}$ cannot be three-valued.
\end{abstract}
\maketitle
\section{Introduction}
In this paper, we are concerned with Weil sums of binomials of the form
\begin{equation}\label{Lawrence}
\Wfdu=\sum_{x \in F} \psi(x^d + u x),
\end{equation}
where $F$ is a finite field of characteristic $p$, the exponent $d$ is a positive integer such that $\gcd(d,\card{\Fu})=1$, the coefficient $u$ is in $\Fu$, and $\psi \colon F \to \C$ is the canonical additive character $\psi(x)=e^{2\pi i \Tr_{F/\Fp}(x)/p}$.
Nontrivial Weil sums of form 
\[
\sum_{x \in F} \psi(a x^m+b x^n),
\]
with $\gcd(m,\card{\Fu})=\gcd(n,\card{\Fu})=1$ can be reparameterized to the form \eqref{Lawrence}.
Such sums and their relatives arise often in number-theory \cite{Kloosterman,Vinogradow,Davenport-Heilbronn,Karatsuba,Carlitz-1978,Carlitz-1979,Lachaud-Wolfmann,Katz-Livne,Coulter,Cochrane-Pinner-2003,Cochrane-Pinner-2011}, and in applications to finite geometry, digital sequence design, error-correcting codes, and cryptography.
See \cite[Appendix]{Katz} on the various guises in which these sums appear in these applications, and for a bibliography.

We fix $F$ and $d$, and consider the values $\Wfdu$ attains as $u$ varies over $\Fu$, but typically ignore the trivial $\Wfd(0)=0$, which is the Weil sum of a monomial.
We say that $\Wfd$ is {\it $v$-valued} to mean that $\card{\{\Wfdu : u \in \Fu\}}=v$.

If $F$ is of characteristic $p$ and $d$ is a power of $p$ modulo $\card{\Fu}$, then $\psi(x^d)=\psi(x)$, so $\Wfdu$ effectively becomes the Weil sum of the monomial $(1+u)x$, so that
\begin{equation}\label{Ronald}
\Wfdu = \begin{cases}
\card{F} & \text{if $u=-1$,} \\
0 & \text{otherwise},
\end{cases}
\end{equation}
and in this case we say that $d$ is {\it degenerate over $F$}.  For nondegenerate $d$, Helleseth \cite[Theorem 4.1]{Helleseth} showed that one obtains more than two values.
\begin{theorem}[Helleseth, 1976]
If $d$ is nondegenerate over $F$, then $\Wfd$ is at least three-valued.
\end{theorem}
Much interest has focused on which choices of $F$ and $d$ make $\Wfd$ precisely three-valued, and ten infinite families have been found (see \cite[Table 1]{Aubry-Katz-Langevin}).
From these, one finds that if $F$ is of characteristic $p$, and if $[F:\Fp]$ is divisible by an odd prime, then there is a $d$ such $\Wfd$ is three-valued.
However, no three-valued examples have ever been found for fields $F$ where $[F:\Fp]$ is a power of $2$.
This prompted the following conjecture \cite[Conjecture 5.2]{Helleseth}.
\begin{conjecture}[Helleseth, 1976]\label{Herman}
If $F$ is of characteristic $p$ with $[F:\Fp]$ a power of $2$, then $\Wfd$ is not three-valued.
\end{conjecture}
Many attempts have been made to test or prove this conjecture in the case where $F$ is of characteristic $2$, and many fruitful discoveries were made in the process \cite{Games,Calderbank-McGuire-Poonen-Rubinstein,McGuire,Charpin,CakCak-Langevin,Feng}.
The $p=2$ case was at last proved in \cite[Corollary 1.10]{Katz}.
Now the $p=3$ case is proved in this paper as a corollary of a new bound on the $p$-divisibility of Weil sums.

For a nonzero integer $n$, the {\it $p$-adic valuation} of $n$, written $\val(n)$, is the largest $k$ such that $p^k \mid n$, and $\val(0)=\infty$.
If we extend this valuation to the cyclotomic field $\Q(e^{2\pi i/p})$ where the Weil sums lie, we may state the first fundamental result on the $p$-divisibility of Weil sums from \cite[Theorem 4.5]{Helleseth}.
\begin{theorem}[Helleseth, 1976]\label{Gloria}
If $F$ is of characteristic $p$, then we have $\val(\Wfdu) > 0$ for every $u \in F$.
\end{theorem}

It was recently proved \cite[Theorems 1.7, 1.9]{Katz} that when $\Wfd$ is three-valued, the values must be rational integers.
\begin{theorem}[Katz, 2012]\label{Imogene}
If $F$ is of characteristic $p$, and if $\Wfd$ is three-valued, then the three values are in $\Z$, one of the values is $0$, and $d \equiv 1 \pmod{p-1}$.
\end{theorem}
So if $\Wfd$ is three-valued, these two theorems show that $p\mid \Wfdu$ for all $u \in F$.  
Our main result is a much stronger lower bound on the $p$-divisibility.
\begin{theorem}\label{Celine}
If $F$ is of characteristic $p$ and order $q=p^n$, and if $\Wfd$ is three-valued with values $0$, $a$, and $b$, then one of the following holds:
\begin{enumerate}[(i).]
\item $\val(a), \val(b) > n/2$; or
\item\label{Arthur} $\val(a)=\val(b)=n/2$, and $|a-b|$ is a power of $p$ with $|a-b| > \sqrt{q}$;
\end{enumerate}
and case \eqref{Arthur} cannot occur if $p=2$ or $3$.
\end{theorem}
For the fields of interest in Conjecture \ref{Herman}, we use the techniques of \cite{Aubry-Katz-Langevin} to show an upper bound on the $p$-divisibility of some $\Wfdu$.
\begin{theorem}\label{Dorothy}
Let $F$ be of characteristic $p$ and order $q=p^n$, with $n$ a power of $2$.
If $\Wfd$ is three-valued, then there is some $u \in \Fu$ such that $\val(\Wfdu) \leq n/2$.
\end{theorem}
This generalizes the result of Calderbank, McGuire, Poonen, and Rubinstein, who proved the $p=2$ case in \cite{Calderbank-McGuire-Poonen-Rubinstein}.
Theorems \ref{Celine} and \ref{Dorothy} immediately combine to prove Conjecture \ref{Herman} in characteristic $2$ and $3$.
\begin{theorem}
If $F$ is of characteristic $p=2$ or $3$ with $[F:\Fp]$ a power of $2$, then $\Wfd$ is not three-valued.
\end{theorem}
In Section \ref{Albert}, we outline a group algebra approach to this problem inspired by the work of Feng \cite{Feng}, and in Section \ref{Bernard} we use a group-theoretic approach of McGuire \cite{McGuire} to determine congruences on the zero counts of critical polynomials that arise in our proofs.
We then prove Theorem \ref{Celine} in Section \ref{Elizabeth}, and prove Theorem \ref{Dorothy} in Section Section \ref{Felicity}.

\section{The Group Algebra and the Fourier Transform}\label{Albert}

This section generalizes the group ring techniques of Feng \cite{Feng} to arbitrary characteristic.
As in the Introduction, $F$ is a finite field.

Consider the group algebra $\gr$ over $\C$, whose elements are written as formal sums $S=\sum_{u \in \Fu} S_u [u]$ with $S_u \in \C$.
We identify any subset $U$ of $\Fu$ with $\sum_{u \in U} [u]$ in $\gr$.
For example, $\Fu$ itself is identified with $\sum_{u \in \Fu}[u]$.
If $S \in \gr$, we let $\card{S}=\sum_{u\in F} S_u$; if $S$ represents a subset of $\Fu$, this is indeed the cardinality of that set.
Note that $S \Fu=\card{S} \Fu$ for any $S \in \gr$.

For $t \in \Z$ and $S=\sum_{u \in \Fu} S_u[u] \in \gr$, we write $S^{(t)}$ to denote $\sum_{u \in \Fu} S_u [u^t]$.
Note that $|S^{(t)}|=\card{S}$.
If $S \in \gr$, its conjugate is defined to be $\conj{S}=\sum_{u \in F} \conj{S}_u [u^{-1}]$, and note that $\card{\conj{S}}=\conj{\card{S}}$.

We let $\mchars$ denote the group of multiplicative characters from $\Fu$ to $\C$.
If $S=\sum_{u \in \Fu} S_u[u] \in \gr$, and $\chi \in \mchars$, we define $\chi(S)=\sum_{u \in \Fu} S_u \chi(u)$.
We call $\chi(S)$ the {\it Fourier coefficient of $S$ at $\chi$}, and we define the {\it Fourier transform of $S$}, denoted $\ft{S}$, to be a function from $\mchars$ to $\C$, where the value of $\ft{S}$ at $\chi$ is $\ft{S}(\chi)=\chi(S)$.

It is straightforward to show that the Fourier transform $S \mapsto \ft{S}$ is an isomorphism of $\C$-algebras from $\gr$ with its convolutional multiplication to the $\C$-algebra $\fr$ of functions from $\mchars$ into $\C$, with pointwise multiplication.
This affords an inverse Fourier transform,
\begin{equation}\label{Oswald}
S_u = \frac{1}{\card{\Fu}} \sum_{\chi \in \mchars} \chi(S) \conj{\chi(u)},
\end{equation}
so that $S=T$ if and only if $\chi(S)=\chi(T)$ for all $\chi\in\mchars$.

Note that if $t$ is an integer, then $\chi(S^{(t)})=\chi^t(S)$.
Similarly, $\chi(\conj{S})=\conj{\chi(S)}$.
If $\chi_0$ is the principal character, then $\chi_0(S)=\card{S}$ for all $S \in \gr$.

As in the Introduction, we let $d$ be a positive integer with $\gcd(d,\card{\Fu})=1$, let $\psi\colon F \to \C$ be the canonical additive character, and we set
\[
\Wu=\Wfdu=\sum_{x \in F} \psi(x^d+ u x).
\]
We are interested in the Fourier analysis of the group algebra element
\begin{equation}\label{Winston}
W=\sum_{u \in \Fu} \Wu [u],
\end{equation}
which records the various values of $\Wfd$ as its coefficients.
To this end, we introduce the group algebra element
\[
\Psi=\sum_{u \in \Fu} \psi(u) [u],
\]
which has a close connection to $W$.
\begin{lemma}\label{Aaron}
$W=\Psi \Psip{-1/d} + \Fu$, where $1/d$ denotes the multiplicative inverse of $d$ modulo $\card{\Fu}$.
\end{lemma}
\begin{proof}
Note that $\Psi \Psip{-1/d} = \sum_{y,z \in \Fu} \psi(y) \psi(z) [y z^{-1/d}]$, but then reparameterize with $z=x^d$ and $y=u x$ to get $\Psi \Psip{-1/d} = \sum_{u,x \in \Fu} \psi(u x) \psi(x^d) [u] = \sum_{u \in \Fu} (\Wu-1) [u]$, so that $\Psi \Psip{-1/d}+\Fu=W$.
\end{proof}
Now we may carry out the Fourier analysis of $\Psi$ and $W$.
\begin{lemma}\label{Eric}
If $\chi \in \mchars$ is not the principal character $\chi_0$, then $\card{\chi(\Psi)}=\sqrt{\card{F}}$, whereas $\chi_0(\Psi)=\card{\Psi}=-1$.
\end{lemma}
\begin{proof}
For any $\chi \in \mchars$, we have
\[
\chi(\Psi)=\sum_{u \in \Fu} \psi(u) \chi(u),
\]
which is a Gauss sum, of which the magnitude in general and the value when $\chi=\chi_0$ are well known (see \cite[Theorem 5.11]{Lidl-Niederreiter}).
\end{proof}
\begin{corollary}\label{Bartholomew}
If $t \in \Z$ with $\gcd(t,\card{\Fu})=1$, then $\Psip t \conj{\Psip t} =\card{F}-\Fu$.
\end{corollary}
\begin{proof}
For $\chi\in\mchars$, we have $\chi(\Psip t \conj{\Psip t})=|\chi(\Psip t)|^2=|\chi^t(\Psi)|^2$, which by Lemma \ref{Eric} and the invertibility of $t$ modulo $\card{\Fu}=|\mchars|$, equals $1$ if $\chi$ is principal, and equals $\card{F}$ otherwise.
One obtains the same values by applying $\chi$ to $\card{F}-\Fu$, so that$\Psip t \conj{\Psip t}=\card{F}-\Fu$ by \eqref{Oswald}.
\end{proof}
\begin{corollary}\label{William}
$W \Wc=\card{F}^2$.
\end{corollary}
\begin{proof}
In view of Lemma \ref{Aaron}, we multiply $W=\Psi \Psip{-1/d} + \Fu$ by its conjugate.
Note that $\conj{\Fu}=\Fu$, recall that $S \Fu=\card{S} \Fu$ for any $S \in \gr$, and use Lemma \ref{Eric} to see that $\card{\Psip{-1/d}}=\card{\Psi}=-1$ and thus $\card{\conj{\Psip{-1/d}}}=\card{\conj{\Psi}}=-1$, so that $W\conj{W}=\Psi \conj{\Psi} \Psip{-1/d} \conj{\Psip{-1/d}} + (\card{\Fu}+2) \Fu$.
Then apply Corollary \ref{Bartholomew} to the first term to finish.
\end{proof}
Now define
\begin{equation}\label{Xavier}
X = \sum_{u \in \Fu} \Wu^2 [u]
\end{equation}
and
\begin{equation}\label{Veronica}
V = \sum_{v \in F} [(v^d+(1-v)^d)^{1/d}], 
\end{equation}
where the $1/d$ is understood modulo $\card{\Fu}$.
We claim that $v^d+(1-v)^d\not=0$ for any $v \in F$, because $v\not=v-1$, and the condition $\gcd(d,\card{\Fu})=1$ makes $x\mapsto x^d$ a permutation of $F$ and makes $d$ odd when the characteristic of $F$ is odd.
This makes $V$ a legitimate element of $\gr$ with
\begin{equation}\label{Fred}
\card{V}=\card{F}.
\end{equation}
We can now relate $W$, $X$, and $V$.
\begin{lemma}\label{Ursula} 
$X = W V.$
\end{lemma}
\begin{proof}
Note that
\begin{align*}
(W V)_u 
& = \sum_{z \in \Fu} W_{u/z} V_z \\
& = \sum_{v \in F} W_{u (v^d+(1-v)^d)^{-1/d}} \\
& = \sum_{v ,w\in F} \psi(w^d+ u (v^d+(1-v)^d)^{-1/d} w) \\
& = \card{F} + \sum_{v \in F} \sum_{w \in F^*} \psi(w^d+ u (v^d+(1-v)^d)^{-1/d} w).
\end{align*}
Since $\gcd(d,\card{\Fu})=1$ makes $d$ odd when $F$ is of odd characteristic, the map $(x,y) \mapsto (x/(x+y),(x^d+y^d)^{1/d})$ is a bijection from $\{(x,y) \in F^2 : x+y\not=0\}$ to $F\times\Fu$, with inverse $(v,w)\mapsto (v^d+(1-v)^d)^{-1/d} (v w,(1-v) w)$.  We may then reparameterize our sum to obtain
\begin{align*}
(W V)_u 
& = \card{F}+\sums{x,y \in F \\ x+y\not=0} \psi(x^d + y^d + u (x+y)) \\
& = \sum_{x,y \in F} \psi(x^d+y^d+u(x+y)) \\
& = \Wu^2.\qedhere 
\end{align*}
\end{proof}
It is now easy to calculate the first four power moments of the Weil sum.
\begin{corollary}\label{Orestes}
We have
\begin{enumerate}[(i).]
\item\label{Katherine} $\sum_{u \in \Fu} \Wu=\card{F}$,
\item\label{Laura} $\sum_{u \in \Fu} \Wu^2=\card{F}^2$, 
\item\label{Mary} $\sum_{u \in \Fu} \Wu^3=\card{F}^2 V_1$, and 
\item\label{Natasha} $\sum_{u \in \Fu} \Wu^4=\card{F}^2 \sum_{u \in \Fu} V_u^2$. 
\end{enumerate}
\end{corollary}
\begin{proof}
First of all, $\sum_{u \in \Fu} \Wu=\card{W}$, which equals $\card{\Psi} |\Psip{-1/d}| + \card{\Fu}=(-1)^2+\card{\Fu}$ by Lemmata \ref{Aaron} and \ref{Eric}.

Secondly, $\sum_{u \in \Fu} \Wu^2=\card{X}$, which equals $\card{W} \cdot \card{V}=\card{F}^2$ by Lemma \ref{Ursula}, equation \eqref{Fred}, and the previous part.

Thirdly, $\sum_{u \in \Fu} \Wu^3$ is the coefficient for the element $[1]$ in $X \Wc$, and $X\Wc=(W V)\Wc=\card{F}^2 V$ by Lemma \ref{Ursula} and Corollary \ref{William}.

Finally, $\sum_{u \in \Fu} \Wu^4$ is the coefficient of $[1]$ in $X \conj{X}$, and we see that $X\conj{X}=(W V)(\conj{W V})=\card{F}^2 V \conj{V}$ by Lemma \ref{Ursula} and Corollary \ref{William}.
\end{proof}

\section{Congruences for $V_1$}\label{Bernard}

As in previous sections, $F$ is a finite field and $d$ is a positive integer with $\gcd(d,\card{\Fu})=1$.
We recall $V$ from equation \eqref{Veronica}, whose coefficient $V_1$ appears in Corollary \ref{Orestes}.\ref{Mary}.
We deduce useful congruences for $V_1$ by means of a group action studied by McGuire \cite{McGuire}.
\begin{proposition}\label{Barbara} Let $V_1=\card{\{v \in F : v^d+(1-v)^d=1\}}$.
\begin{enumerate}[(i).]
\item\label{Herbert} If $\card{F} \equiv 0 \pmod{3}$, then $V_1 \equiv 3 \pmod{6}$.
\item If $\card{F} \equiv 1 \pmod{3}$, then
\begin{enumerate}[(a).]
\item if $2^{d-1}=1$ in $F$, then $V_1 \equiv 1 \pmod{6}$, but
\item if $2^{d-1}\not=1$ in $F$, then $V_1 \equiv 4 \pmod{6}$.
\end{enumerate}
\item If $\card{F} \equiv 2 \pmod{3}$, then 
\begin{enumerate}[(a).]
\item if $2^{d-1}=1$ in $F$, then $V_1 \equiv 5 \pmod{6}$, but
\item if $2^{d-1}\not=1$ in $F$, then $V_1 \equiv 2 \pmod{6}$.
\end{enumerate}
\end{enumerate}
\end{proposition}
Before proving the proposition, we interject two remarks.
\begin{remark}\label{Caesar}
When $F$ is of order $2^n$, then we have $2=0$ in $F$, so that $2^{d-1}\not=1$, and we recover the result of McGuire (cf.~\cite[Corollary 1]{McGuire}) that $V_1 \equiv 4 \pmod{6}$ if $n$ is even, and $V_1 \equiv 2 \pmod{6}$ if $n$ is odd.
\end{remark}
\begin{remark}
When we are dealing with three-valued Weil sums, Theorem \ref{Imogene} shows that $d\equiv 1 \pmod{p-1}$, in which case $2^{d-1}=1$ in $F$ by Fermat's Little Theorem whenever $p\not=2$.
\end{remark}
\noindent {\it Proof of Proposition \ref{Barbara}.}
We are counting modulo $6$ the roots of $f(X)=X^d+(1-X)^d-1$ in $F$.
First of all, note that $0$ and $1$ are always roots of $f$.
Consider the action of the two involutions $\sigma(x)=1-x$ and $\tau(x)=1/x$ on $F\smallsetminus\{0,1\}$.
Note that $f(\sigma(x))=f(x)$ and $f(\tau(x))=-f(x)/x^d$, because $\gcd(d,\card{\Fu})=1$ makes $d$ odd when $\card{F}$ is odd.
Thus $\sigma$ and $\tau$ map the roots of $f$ in $F\smallsetminus\{0,1\}$ to themselves.

For generic $F$, our $\sigma$ and $\tau$ generate a group $G$ isomorphic to $S_3$, consisting of the identity, $\sigma$, $\tau$, $\tau \circ \sigma(x)=1/(1-x)$, $\sigma \circ \tau(x)=(x-1)/x$, and $\sigma \circ \tau \circ \sigma(x)=\tau \circ \sigma \circ \tau(x)=x/(x-1)$.
The set of roots of $f$ in $F\smallsetminus\{0,1\}$ is preserved by $G$, so is partitioned into orbits under its action.

A $G$-orbit contains $6$ elements unless it contains a point fixed by some nonidentity element of $G$.
An $x \in F\smallsetminus\{0,1\}$ is fixed by a nonidentity element of $G$ if and only if it satisfies at least one of the following equations: $x=1-x$, $x=1/x$, $x=1/(1-x)$, $x=(x-1)/x$, or $x=x/(x-1)$.
Equivalently, $x$ equals at least one of: (i) $-1$, (ii) $2$, (iii) the multiplicative inverse of $2$ (which we shall call $1/2$ when it exists), or (iv) a root of $\Phi_6(X)=X^2-X+1$, the cyclotomic polynomial of index $6$.

If $\card{F}\equiv 0\pmod{3}$, then $-1=2=1/2$, which is the double root of $\Phi_6$.
Furthermore, $f(-1)=0$ since $d$ is odd.
Thus the roots of $f$ consist of $0$, $1$, $-1$, and orbits of size $6$, so that $V_1\equiv 3 \pmod{6}$.

If $F$ is of characteristic $2$, then (i) $-1=1$, (ii) $2=0$, and (iii) the multiplicative inverse of $2$ does not exist, so these are not points in $F\smallsetminus\{0,1\}$, and (iv) the roots of $\Phi_6$ are the primitive third roots of unity, $\omega$ and $\omega^{-1}=1-\omega$, which lie in $F$ if and only if $\card{F} \equiv 1 \pmod{3}$.
In this case $f(\omega)=f(\omega^{-1})=0$ since $d \equiv \pm 1 \pmod{3}$, because $\gcd(d,\card{\Fu})=1$.
So the roots of $f$ consist of $0$ and $1$ always, the third roots of unity if $\card{F}\equiv 1 \pmod{3}$, and orbits of size $6$.
Thus $V_1 \equiv 4 \pmod{6}$ if $\card{F} \equiv 1 \pmod{3}$ and $V_1 \equiv 2 \pmod{6}$ if $\card{F} \equiv 2 \pmod{3}$.

Now suppose that $F$ is of characteristic $p \geq 5$.
Then (i) $-1$, (ii) $2$, (iii) $1/2$, and (iv) the roots of $\Phi_6$ are five distinct elements in characteristic $p$.
The roots $\zeta$ and $\zeta^{-1}=1-\zeta$ of $\Phi_6$ lie in $F$ if and only if $\card{F} \equiv 1 \pmod{6}$, or equivalently, if and only if $\card{F} \equiv 1 \pmod{3}$.
In this case $f(\zeta)=f(\zeta^{-1})=0$ since $d \equiv \pm 1 \pmod{6}$ because $\gcd(d,\card{\Fu})=1$, so the roots of $\Phi_6$ are roots of $f$ if they are present in $F$.
The three elements $-1$, $2$, and $1/2$ make up a $G$-orbit, and are roots of $f$ if and only if $0=f(1/2)=(1/2)^d+(1/2)^d-1=(1/2)^{d-1}-1$, that is, if and only if $2^{d-1}=1$.
So the roots of $f$ consist of $0$ and $1$ always, the roots of $\Phi_6$ if and only if $\card{F}\equiv 1 \pmod{3}$, the elements $-1$, $2$, and $1/2$ if and only if $2^{d-1}=1$ in $F$, and orbits of size $6$.
Adding these counts together modulo $6$ according to the four possible cases finishes the proof.\hfill\qed

\section{Proof of Theorem \ref{Celine}}\label{Elizabeth}

Throughout this section, we assume that $F$ is a finite field of characteristic $p$ and order $q$, and that $d$ is a positive integer with $\gcd(d,q-1)=1$.
We let $\psi\colon F \to \C$ be the canonical additive character, and set
\[
\Wu=\Wfdu=\sum_{x \in F} \psi(x^d+ u x),
\]
and assume that $\Wfd$ is three-valued.
Per Theorem \ref{Imogene}, these three values must all be in $\Z$, and one of them must be $0$.
Call the other two values $a$ and $b$; these are of opposite sign since $\sum_{u \in \Fu} \Wu^2 = \left(\sum_{u \in \Fu} \Wu\right)^2$ by Corollary \ref{Orestes}.\ref{Katherine}--\ref{Laura}.
We calculate the power moments in this three-valued case.
\begin{lemma}\label{Theresa}
For any positive integer $k$, we have
\[
\sum_{u \in \Fu} \Wu^k = \frac{q^2(a^{k-1}-b^{k-1}) - q a b(a^{k-2}-b^{k-2})}{a-b}.
\]
\end{lemma}
\begin{proof}
The $k=1$ and $2$ cases are proved in Corollary \ref{Orestes}.\ref{Katherine}--\ref{Laura}.
For $k > 2$, we note that $\Wu^{k-2} (\Wu-a) (\Wu-b) = 0$ for all $u \in \Fu$, whence
\[
\sum_{u \in \Fu} \Wu^k = (a+b) \sum_{u \in \Fu} \Wu^{k-1} - a b \sum_{u \in \Fu} \Wu^{k-2},
\]
from which the identity for $\sum_{u \in \Fu} W_u^k$ follows by induction.
\end{proof}
We let $V$ be as defined in equation \eqref{Veronica}, and note the consequences that our power moment results have for the coefficients of $V$.
\begin{lemma}\label{Victor}
$V_1=a+b-\frac{a b}{q}$, hence $\val(a b) \geq \val(q)$, with strict inequality if $p=2$ or $3$.
\end{lemma}
\begin{proof}
Compare Corollary \ref{Orestes}.\ref{Mary} and the $k=3$ case of Lemma \ref{Theresa} to deduce the equation.
Since $V_1=\card{\{v\in F : v^d+(1-v)^d=1\}}$, and since $a, b \in \Z$, we know that $a b/q$ lies in $\Z$, so that $\val(a b) \geq \val(q)$.
Suppose that $F$ is of characteristic $p=2$ or $3$.
Since $\Wfd$ is three-valued, $d \equiv 1 \pmod{p-1}$ by Theorem \ref{Imogene}, and so by Proposition \ref{Barbara}.\ref{Herbert} and Remark \ref{Caesar}, we see that $V_1 \equiv 0 \pmod{p}$.
Furthermore, $a \equiv b \equiv 0 \pmod{p}$ by Theorem \ref{Gloria}.
Thus $a b/q \equiv 0 \pmod{p}$, so $\val(a b) > \val(q)$.
\end{proof}
\begin{lemma}\label{Wilbur}
We have
\begin{enumerate}[(i).]
\item\label{John} $\sum_{u \in \F\smallsetminus\{0,1\}} V_u = \frac{(q-a)(q-b)}{q} >0$, and
\item\label{Yolanda} $\sum_{u \in \F\smallsetminus\{0,1\}} V_u^2 = -\frac{a b (q-a)(q-b)}{q^2}>0$.
\end{enumerate}
\end{lemma}
\begin{proof}
For part \eqref{John}, we have $\sum_{u \in \F\smallsetminus\{0,1\}} V_u=\card{V}-V_1$, and then use equation \eqref{Fred} and the value of $V_1$ from Lemma \ref{Victor}.  For part \eqref{Yolanda}, compare Corollary \ref{Orestes}.\ref{Natasha} and the $k=4$ case of Lemma \ref{Theresa} to see that
\[
q^2 \sum_{u \in \Fu} V_u^2 = q^2(a^2+a b+b^2) - q a b (a+b), 
\]
and then substitute the value of $V_1$ from Lemma \ref{Victor} and rearrange.
Both our sums are strictly positive since $|a|, |b| < q$, inasmuch as Corollary \ref{Orestes}.\ref{Laura} shows that $q^2=\sum_{u \in \Fu} \Wu^2$, which must be at least $a^2+b^2$.
\end{proof}
Now write $|a|=a_o a_p$, $|b|=b_o b_p$, and $|a-b|=(a-b)_o (a-b)_p$, where  $a_p$, $b_p$, and $(a-b)_p$ are powers of $p$ and $a_o$, $b_o$, and $(a-b)_o$ are positive integers coprime to $p$.
So $\gcd(a_o,b_o)$ is coprime to $p$, and yet also a divisor of every $\Wu$, hence of $\card{W}=q$ (see Corollary \ref{Orestes}.\ref{Katherine}), and so $\gcd(a_o,b_o)=1$.
Thus $a_o$, $b_o$, and $(a-b)_o$ are pairwise coprime, and now we show that they divide most of the coefficients of $V$.
\begin{lemma}\label{Zachary}
$a_0 b_o (a-b)_0 \mid V_u$ for all $u\in \Fu$ with $u\not=1$.
\end{lemma}
\begin{proof}
Recall the definitions of $W$ and $X$ in equations \eqref{Winston} and \eqref{Xavier}.  
Let $A=\{u \in \Fu: \Wu=a\}$ and $B=\{u \in \Fu: \Wu=b\}$, so that $W = a A + b B$ and $X=a^2 A + b^2 B$.
Since $X=W V$ by Lemma \ref{Ursula}, we solve for $A$ to get $a(a-b) A = W (V-b)$, then multiply both sides by $\Wc$ and apply Corollary \ref{William} to get $a(a-b) A \Wc = q^2 (V-b)$.
Since $A$ and $\Wc$ have integer coefficients, $a_o(a-b)_o$ divides all the coefficients of $V-b$, in particular, all $V_u$ with $u\not=1$.
By the same method, we deduce that $b_o(a-b)_o$ divides all $V_u$ with $u\not=1$, and recall that $a_o$, $b_o$, and $(a-b)_o$ are pairwise coprime.
\end{proof}
Now we recall and prove our main result, Theorem \ref{Celine}.
\begin{theorem}
If $q=p^n$, then one of the following holds:
\begin{enumerate}[(i).]
\item $\val(a), \val(b) > n/2$; or
\item\label{Benjamin} $\val(a)=\val(b)=n/2$, and $|a-b|$ is a power of $p$ with $|a-b| > \sqrt{q}$;
\end{enumerate}
and case \eqref{Benjamin} cannot occur if $p=2$ or $3$.
\end{theorem}
\begin{proof}
By Lemma \ref{Zachary}, $V_u^2 \geq a_o b_o (a-b)_o V_u$ for each $u\not=0,1$.  Thus 
\[
\sum_{u \in \F\smallsetminus\{0,1\}} V_u^2 \geq a_o b_o (a-b)_o \sum_{u \in \F\smallsetminus\{0,1\}} V_u,
\]
so by Lemma \ref{Wilbur} we can divide by the sum on the right to obtain $-a b/q \geq a_o b_o (a-b)_o$, which yields
\begin{equation}\label{Alexandra}
a_p b_p \geq q (a-b)_o.
\end{equation}

If $a_p\not=b_p$, set $\{g,h\}=\{a_p,b_p\}$ with $g < h$.
Then $(a-b)_p=g$, so $(a-b)_o=|a-b|/g$, and thus \eqref{Alexandra} becomes $g^2 h \geq q |a-b|$.
Now $|a-b| > \max\{|a|,|b|\} \geq h$, so that $g^2 > q$, and so $a_p, b_p > \sqrt{q}$, and hence $\val(a), \val(b) > n/2$.

If $a_p=b_p$, then $\val(a b)\geq\val(q)$ by Lemma \ref{Victor}.
If this inequality is strict, which it must be if $p=2$ or $3$, then $\val(a)=\val(b) > n/2$.

So it remains to consider the case where $p\geq 5$ and $\val(a)=\val(b)=n/2$.
Then $a_p b_p=q$, so \eqref{Alexandra} forces $(a-b)_0=1$, hence $|a-b|$ must be a power of $p$, and indeed must be greater than $\sqrt{q}$ since $|a-b| > |a|$ and $a$ is a nonzero integral multiple of $\sqrt{q}$.
\end{proof}
We conclude with a remark about what happens when we are in case \eqref{Benjamin} of our theorem, and what that tells us about possible counterexamples to Conjecture \ref{Herman}.
\begin{remark}
In case \eqref{Benjamin} of our theorem, where $\val(a)=\val(b)=n/2$, we have $-a b/q=a_o b_o$, so that $V_1=a+b+a_o b_o$ by Lemma \ref{Victor}, and $\sum_{v \in F\smallsetminus\{0,1\}} V_u(V_u-a_o b_o) = 0$ by Lemma \ref{Wilbur}.
Since $V_u$ is equal to the count $\card{\{v \in F : v^d+(1-v)^d=u^d\}}$, and since Lemma \ref{Zachary} tells us that $a_o b_o \mid V_u$ for all $u \in F\smallsetminus\{0,1\}$, we see that $V_u \in \{0,a_o b_o\}$ for all $u \in F\smallsetminus\{0,1\}$.  Furthermore, $|a-b|$ is a power of $p$ greater than $\sqrt{q}$, so that $a_o+b_o$ is also a power of $p$, and $a_o+b_o \geq p \geq 5$, and so $a_o b_o \geq 4$.

Theorem \ref{Dorothy} forces us into this $\val(a)=\val(b)=n/2$ case when $n$ is a power of $2$.
So if there is a counterexample to Conjecture \ref{Herman}, the field $F$ and exponent $d$ must have the property that $v^d+(1-v)^d$ represents $1$ for precisely $a+b+a_o b_o$ values of $v \in F$, and it represents $\frac{q-a-b-a_o b_o}{a_o b_o}$ elements of $F\smallsetminus\{0,1\}$ for precisely $a_o b_o$ values of $v \in F$ each, and it does not represent any other element of $F$. 
\end{remark}

\section{Proof of Theorem \ref{Dorothy}}\label{Felicity}

Throughout this section, we use the definition of $\Wfdu$ from \eqref{Lawrence}.
We prove Theorem \ref{Dorothy} using two results of Aubry, Katz, and Langevin \cite{Aubry-Katz-Langevin}.
The first key result is Corollary 4.2 of that paper.
\begin{proposition}[Aubry-Katz-Langevin, 2013] \label{George}
Let $K$ be a finite field and $L$ an extension of $K$ of finite degree, and suppose that $d$ is a positive integer with $\gcd(d,\card{L^*})=1$.
Then
\[
\min_{u \in L^*} \val(W_{L,d}(u)) \leq [L:K] \cdot \min_{u \in K^*} \val(W_{K,d}(u)).
\]
\end{proposition}
The second result is part of Corollary 4.4 of the same paper.
\begin{proposition}[Aubry-Katz-Langevin, 2013]\label{Henry}
Let $K$ be a finite field of characteristic $p$, and let $L$ be the quadratic extension of $K$.
Let $d$ be a positive integer with $\gcd(d,\card{L^*})=1$, and suppose that $d$ is degenerate over $K$ but not over $L$.
Then
\[
\min_{u \in L^*} \val(W_{L,d}(u)) = [K:\Fp].
\]
\end{proposition}
We now recall and prove Theorem \ref{Dorothy}.
\begin{theorem}
Let $F$ be a field of characteristic $p$ and order $q=p^n$, with $n$ a power of $2$.
Let $d$ be a positive integer with $\gcd(d,q-1)=1$ such that $\Wfd$ is three-valued.
Then there is some $u \in \Fu$ such that $\val(\Wfdu) \leq n/2$.
\end{theorem}
\begin{proof}
By Theorem \ref{Imogene}, $d$ is degenerate over $\Fp$.
But $d$ is not degenerate over $F$ since $\Wfd$ is three-valued (cf.~\eqref{Ronald}).
Proceeding by successive quadratic extensions from $\Fp$ to $F$, there must be subfields, say $K$ and $L$, of $F$ with $[L:K]=2$ and $d$ degenerate over $K$ but not $L$.
Then by Propositions \ref{George} and \ref{Henry}, we have
\begin{align*}
\min_{u \in \Fu} \val(\Wfdu) 
& \leq [F:L] \cdot \min_{u \in L^*} \val(W_{L,d}(u)) \\
& = [F:L] [K:\Fp],
\end{align*}
so that there is some $u \in \Fu$ with $\val(\Wfdu) \leq [F:\Fp]/[L:K] = n/2$.
\end{proof}

\section*{Acknowledgement}

The author was supported in part by a Research, Scholarship, and Creative Activity Award from California State University, Northridge.

\end{document}